\documentclass[12pt]{amsart}
\usepackage{amscd,amssymb}
\usepackage[arrow,matrix]{xy}
\usepackage[plainpages,backref,urlcolor=blue]{hyperref}

\theoremstyle{plain}
\newtheorem{thm}[subsection]{Theorem}
\newtheorem{lem}[subsection]{Lemma}
\newtheorem{prop}[subsection]{Proposition}
\newtheorem{cor}[subsection]{Corollary}

\theoremstyle{definition}
\newtheorem{rk}[subsection]{Remark}

\numberwithin{equation}{section}
\newcommand{\OO}{{\mathcal O}}

\newcommand{\F}{{\mathcal F}}

\newcommand{\wX}{\widehat{X}}
\newcommand{\wi}{\widehat{i}}
\newcommand{\wL}{\widehat{L}}
\newcommand{\wV}{\widehat{V}}
\newcommand{\wY}{\widehat{Y}}

\newcommand{\al}{{\alpha}}
\newcommand{\be}{{\beta}}

\newcommand{\Z}{\mathbb{Z}}
\newcommand{\Q}{\mathbb{Q}}
\newcommand{\R}{\mathbb{R}}
\newcommand{\C}{\mathbb{C}}

\newcommand{\bL}{\mathbb{L}}
\newcommand{\PP}{\mathbb{P}}
\newcommand{\HH}{\mathbb{H}}
\newcommand{\N}{\mathbb{N}}

\newcommand{\g}{\mathfrak {g}}

\DeclareMathOperator{\Hom}{Hom}

\DeclareMathOperator{\im}{im}

\DeclareMathOperator{\codim}{codim}

\DeclareMathOperator{\gr}{Gr}

\begin{document}

\title [On the fundamental groups of normal varieties]
{On the fundamental groups  of normal varieties }

\author[Donu Arapura]{Donu Arapura$^1$ }
\thanks{$^1$Partially supported by the NSF and at the IAS by the Ralph E. and
  Doris M. Hansmann fund.}
\address{ Department of Mathematics, Purdue University, West Lafayette, IN 47907, U.S.A.}
\email{dvb@math.purdue.edu}

\author[Alexandru Dimca]{Alexandru Dimca$^2$}
\address{Univ. Nice Sophia Antipolis, CNRS,  LJAD, UMR 7351, 06100 Nice, France. }
\email{dimca@unice.fr}

\thanks{$^2$ Partially supported by Institut Universitaire de France and IAS Princeton} 

\author[Richard Hain]{Richard Hain$^3$}
\address{Department of Mathematics, Duke University, Durham, NC 27708-0320, U.S.A.
 }
\email{hain@math.duke.edu }

\thanks{$^3$ Supported in part by grant DMS-1406420 from the National Science
Foundation and by the Friends of the Institute for Advanced Study.}



\begin{abstract} We show that the fundamental groups of normal complex algebraic varieties share
 many properties of the fundamental groups of smooth varieties.
The jump loci of rank one local systems on a normal variety are related to the jump loci of a resolution and of a smoothing of this variety.
\end{abstract}

\maketitle

\section{Introduction } \label{s0}

By recent work of M. Kapovich, J. Koll\'ar and C. Simpson it is known that any finitely presented group is the fundamental group of an irreducible projective variety; moreover, the  singularities can be chosen to be of a specific type, see \cite{Ka}, \cite{KaKo}, \cite{Si}. On the other hand, a fundamental group $G$ of a smooth projective or quasi-projective variety satisfies a number of special properties.
For instance, if $G=\pi_1(X,x)$, with $X$ quasi-projective smooth and $x \in X$, then it was shown by J. Morgan that
$\g_{\C}$, the Malcev Lie algebra of $G$ over $\C$ has a natural mixed Hodge structure such that the associated graded Lie algebra $\gr^W_*\g_{\C}$ with respect to the weight filtration $W$, is isomorphic to ${\bL}/J$, where $\bL$ is a free 
Lie algebra with generators in degrees $-1$ and $-2$,  
and $J$ is a  Lie ideal, generated in degrees 
$-2$, $-3$ and $-4$, see \cite[Corollary 10.3]{Mo} and \cite[Theorem 5.8]{H3}.  For smooth quasi-projective varieties, the list of known restrictions is much longer, see for instance  \cite{ABC}, \cite{DPS},  \cite{DPSbull} and the references therein.

In this note, we consider the case of normal varieties. In a nutshell, we want to give some evidence that  {\it the fundamental groups of normal varieties behave like those of smooth varieties}; see also the remarks at the end of section 2.3 in the first chapter of \cite{ABC}.
In his 1974 ICM talk, P. Deligne stated in \cite{De}, section 10, that Morgan's theorem holds for $X$ a normal variety. 
We include a proof of Deligne's result. When combined with \cite{AN}, it puts strong restrictions on the solvable groups in the above class. 
We also bring additional support to the above feeling, 
by showing that all the specific properties of 
the fundamental groups of smooth varieties reflected by their resonance and characteristic varieties
as explained in \cite{DPS} extends to the larger class of normal varieties, see Corollary \ref{maincor}. 
Using this result, we show in Corollary \ref{corappl} that a right-angled Artin group
 is the fundamental group of a normal variety $X$ if and only if it is so for a smooth variety $X$.

\medskip

Here is a detailed description of our results.
For any finite connected CW complex $M$ we can define the first resonance varieties $R^1_k(M) \subset H^1(M,\C)$ as the set of $u \in H^1(M,\C)$ for which the corresponding Aomoto complex $H^*(M,\C), u\wedge)$ satisfies the condition 
$$\dim H^1(H^*(M,\C), u\wedge) \geq k.$$
Similarly, the first characteristic variety $\Sigma^1_k(M) =V^1_k(M) \subset H^1(M,\C^*)$
is the set of rank one local systems $L$ (given by a representation in $\Hom (\pi_1(M), \C^*)=H^1(M,\C^*)$) such that 
$$\dim H^1(M,L) \geq k.$$
 The {\it restricted characteristic variety} of $M$, denoted by $V^1_k(M)_1$, is the union of all irreducible components of the algebraic variety $V^1_k(M)$ which have  strictly positive dimension and pass through the unit element $1$. 
Recall that $R^1_k(M)$, $V^1_k(M)$ and $V^1_k(M)_1$ depend in fact only on the fundamental group $G=\pi_1(M)$ and hence may also be denoted by $R^1_k(G)$, $V^1_k(G)$ and $V^1_k(G)_1$. For simplicity, we set $R^1(M)=R^1_1(M)$, $V^1(M)=V^1_1(M)$ and $V^1(M)_1=V^1_1(M)_1$.

If we apply these constructions to an irreducible normal variety $X$ of dimension $n\geq 2$ and to a resolution of singularities $p:\wX \to X$, we get pairs of resonance varieties $R^1_k(X)$ and $R^1_k(\wX)$, and pairs of  characteristic varieties $V^1_k(X)$ and $V^1_k(\wX)$.
The first main result of this note is  the following generalization of a main theorem in
\cite{A}, and can be used to understand the relation among these varieties associated to $X$ and $\wX$.

\begin{thm} \label{thm2} 
 For any irreducible component $V$ of $V^1_k(X)$, with $ k>0$, having strictly positive dimension, there is a  surjective morphism $f:X\to C$ with  connected generic fibre to a smooth
  curve $C$ with negative Euler characteristic $E(C)\leq 0$ and a torsion character $\chi \in H^1(X,\C^*)$, such that $V$ is
 the translated affine torus $\chi f^*H^1(C,\C^*)$.  Moreover, $p^*(V)$ is an irreducible component of $V^1_k(\wX)$. Conversely, for an irreducible component $\wV$ of $V^1_k(\wX)$ satisfying the condition

\medskip

$(T)$     for any $ L \in \wV$ and $x \in X$   the restriction $ L|F_x $ is trivial,
where $F_x=p^{-1}(x)$, 

\medskip

\noindent there is an irreducible component $V$ of $V^1_k(X)$ such that
$p^*(V)=\wV.$
Finally, the zero dimensional components of $V^1_k(X)$  consist of
torsion characters.
\end{thm}

The second main result of this note is  the following property, informally called Morgan's obstruction and announced at the beginning of this Introduction.

\begin{thm} \label{thm3} 
The graded Lie algebra $\gr^W_*\g$ of the (complex) Malcev Lie algebra $\g=\g(X,x)$ associated to $\pi_1(X,x)$ for a normal irreducible variety $X$ admits a presentation with generators of degree $-1$ and $-2$ and relations of degree $-2,-3$ and $-4$. If $X$ is in addition projective, then the generators can be chosen only of degree $-1$ and the relations only of degree $-2$.

\end{thm}

\begin{cor} \label{maincor2} 
The fundamental group $G=\pi_1(X)$ of a normal projective variety is 1-formal.
\end{cor} 

If $\Gamma$ is a finitely nilpotent group, we will say that it is quadratically presented if its Malcev Lie algebra has generators and relations in degree $-1$ and $-2$ respectively. The class of such groups is strongly restricted, see  \cite{CT}.
By the theorem  \ref{thm3} together with \cite[thm 3.3]{AN}, we deduce.

\begin{cor} \label{maincor3} 
If $X$ is normal and projective such that $\pi_1(X)$ is a solvable subgroup of $GL_n(\Q)$, then $\pi_1(X)$ contains a quadratically presented nilpotent group of finite index.
\end{cor} 

The next result follows by using Theorems \ref{thm2} and   \ref{thm3} and applying Theorem C in \cite{DPS} to $\wX$. We recall that a linear subspace $E \subset H^1(X,\C)$ is said to be $p$-isotropic if the cup-product restriction morphism $E\otimes E \to H^2(X,\C)$ has a $p$-dimensional image and, for $p=1$, induces a pairing on $E$, see \cite{DPS}.

\begin{cor} \label{maincor} 
Let $X$ be an irreducible normal variety. Set $G=\pi_1(X)$ and let $V^{\al}$ be the irreducible components in $V^1_k(X)_1=V^1_k(G)_1$ for some $k>0$. Denote by $T^{\al}$ the tangent space at 1 to $V^{\al}$. Then the following hold.

\begin{enumerate} 
\item Any tangent space $T^{\al}$ is a $p$-isotropic linear subspace of $H^1(X,\C)=H^1(G,\C)$, defined over $\Q$ and of dimension at least $2p+2$, for some $p=p(\al) \in \{0,1\}$.

\item If $\al \ne \be$, then $T^{\al} \cap T^{\be}=0$.

\item Suppose $G$ is a 1-formal group (e.g.  $X$ is in addition projective) and let $R^{\al}$ be the irreducible components of the resonance variety $R^1_k(X)=R^1_k(G)$. Then the  collection $T^{\al}$ coincides with the collection $R^{\al}$. In other words, in this case $R^1_k(X)$ is the tangent cone to $V^1_k(X)$ at the unit element 1.

\end{enumerate}

\end{cor}

As an application, we get the following analog of Theorem 11.7 in \cite{DPS} (with exactly the same proof), telling us which right-angled Artin groups are fundamental groups of normal varities. Recall that all right-angled Artin groups are 1-formal and hence they all pass Morgan's obstruction recalled above.

\begin{cor} \label{corappl} 
Let $\Gamma=(V,E)$ be a finite simplicial graph, with associated right-angled Artin group $G_{\Gamma}$. The following are equivalent.
\begin{enumerate} 
\item There is a connected smooth algebraic variety $X$ such that $\pi_1(X)=G_{\Gamma}$.

\item There is a normal (or unibranch) algebraic variety $X$ such that $\pi_1(X)=G_{\Gamma}$.

\item The graph $\Gamma$ is a complete multipartite graph.

\item The group $G_{\Gamma}$ is a finite product of finitely generated free groups.

\end{enumerate} 
If $X$ is normal and projective, then $\pi_1(X)=G_{\Gamma}$ is a free abelian group of even rank.

In particular, there are infinitely many right-angled Artin groups which are not isomorphic to fundamental groups of normal (or unibranch)  varieties.

\end{cor}

Finally, we explain how to compute characteristic varieties of
normal projective varieties, in some cases, by smoothing them. 
Let $X$ be an  analytic space equipped with a projective holomorphic
map $f:X\to \Delta$ to the open disk which is smooth over $\Delta^*=
\Delta-\{0\}$ and such that  $X_0= f^{-1}(0)$ is normal. Then $X$ is normal  \cite[6.8.1]{EGA}.
Since it is convenient to allow a boundary, we replace $\Delta$, and accordingly $X$, by a smaller closed disk of radius, say $\epsilon$.
Let us fix  a holomorphic section
$\sigma:\Delta\to X$ (which is not a restriction).  We keep this notation throughout this section.
It follows from our assumptions that $X^*= X-X_0\to\Delta^*$ is a fibre bundle. We can also assume, after shrinking $\Delta$ if necessary, that we have a deformation retraction of $X$ to $X_0$ which collapses $\sigma(\Delta)$ to $\sigma(0)$. It follows that we have a specialization map
\begin{equation}
  \label{eq:spPi}
\pi_1(X_t, \sigma(t))\to \pi_1(X_0, \sigma(0))
\end{equation}
for any $t\in \partial\Delta$. The main results in this setting, comparing the invariants associated to the smooth general fiber $X_t$ to those of the normal special fiber $X_0$, are the following.

\begin{prop}\label{prop:spPi}
  The map \eqref{eq:spPi} is surjective.
\end{prop}

\begin{thm}\label{thm:smoothing1}
  With the same assumptions as above, with $t\in \partial\Delta$, the
  following holds.
  \begin{enumerate}
  \item The specialization map induces an isomorphism of vector spaces
$$H^1(X_0,\Z)\cong H^1(X_t,\Z).$$
\item The monodromy acts trivially on $H^1(X_t,\Z)$.
\item The specialization map induces an injection of affine algebraic groups
$$H^1(X_0,\C^*)\hookrightarrow H^1(X_t,\C^*)$$
which is an isomorphism on the identity components, or equivalently the
cokernel is finite.
\item Viewing (3) as an inclusion, we have
$$V^1(X_0)= V^1(X_t)\cap H^1(X_0,\C^*).$$

  \end{enumerate}
\end{thm}
It follows from (4) above that each strictly positive dimensional irreducible component $V_i$ in $V^1(X_0)$ gives rise to a component $V_{i,t}$ in $V^1(X_t)$, which in turn, by \cite{B}, corresponds to a map $X_t \to C_{i,t}$ onto a smooth projective curve $C_{i,t}$ for any $t \in \Delta$. The next result shows that this family of maps $X_t \to C_{i,t}$ can be chosen to depend algebraically on $t \in \Delta$.
\begin{thm}\label{thm:smoothing2}
With the same assumptions as above,
  there exists a finite collection torsion  characters $\chi_i\in H^1(X_0,\C^*),\,
  i=1,\ldots n,\ldots N$, and  commutative diagrams
$$
\xymatrix{
 X\ar[rd]_{f}\ar[r]^{\pi_i} & C_i\ar[d]^{p_i} \\ 
  & \Delta
}
$$
where  $C_i\to \Delta,i=1,\ldots n$ are
relative smooth projective curves, such that
$$V^1(X_t)\cap H^1(X_0,\C^*) =\bigcup_{i=1}^n
\chi_i\,\pi_{i,t}^*H^1(C_{i,t},\C^*)\cup \bigcup_{i=n+1}^N\{\chi_i\}$$
for all $t\in \Delta$.
\end{thm}

It would be interesting to find an example of a group $G$ which is the fundamental group of a normal variety, but not the fundamental group of a smooth variety.

\medskip

We are all very grateful to the Institute of Advanced Study in Princeton for the excellent environment provided during the work on this project. We also thank J\'anos Koll\'ar for drawing our attention to  the references \cite{FL} and  \cite{KoSa}.

\section{A basic fact on fundamental groups of normal analytic varieties} \label{s1}

Let $X$ be an $n$-dimensional irreducible  complex normal analytic variety, with $n \geq 2$, $X_{reg}$ the open subset of regular points of $X$ and $p: \wX \to X$ a resolution of singularities for $X$ such that $p: p^{-1}(X_{reg}) \to X_{reg}$ is an isomorphism. We use this isomorphism to identify $X_{reg}$ to an open subset of $\wX$ and denote by 
$i:X_{reg} \to X$ and $\wi:X_{reg} \to \wX$ the corresponding inclusions, such that $i=p\circ \wi$.
 The first key result is the following, see for a proof (0.7) (B) in \cite{FL} or Proposition 2.10 in \cite{KoSa}.

\begin{thm} \label{thm1} 
The morphisms $p_{\sharp}:\pi_1(\wX) \to \pi_1(X)$ and $i_{\sharp}:\pi_1(X_{reg}) \to \pi_1(X)$ 
are surjective. Moreover, for any proper closed analytic subset $A \subset X$, there is a surjection
$\pi_1(X \setminus A) \to \pi_1(X)$ induced by the inclusion.
\end{thm}

As remarked in  (0.7) (B) in \cite{FL}, the condition on $X$ being normal in the above result can be relaxed by asking that $X$ is {\it unibranch}, i.e. for any point $x \in X$, the analytic germ $(X,x)$ is irreducible. This class strictly contains the class of normal varieties, e.g. consider the product between a cuspidal curve and a smooth variety, but leads to the same family of topological spaces. Indeed, $X$ is a unibrach variety if and only if the normalization morphism $\nu: X' \to X$ is a homeomorphism (in both classical and Zariski topologies). However, in practice, it might be easier to check that a given variety is unibranch rather than normal. We leave the reader to reformulate the following results for unibranch varieties, and assume from now on that $X$ is normal unless stated otherwise.

\begin{rk} \label{rk1} 
(i)  When $n=2$, an alternative proof of Theorem \ref{thm1} can be obtained using the following result, which is a special case of Proposition 3.B.2 in \cite{BK}.

Let $X$ and $Y$ be two irreducible complex normal surfaces and $f:Y \to X$ a surjective proper morphism. If all the fibers of $f$ are connected (resp. simply-connected), then 
the morphism $f_{\sharp}:\pi_1(Y) \to \pi_1(X)$ is surjective (resp. an isomorphism). When $f$ is a resolution of singularities for $X$, the fibers of $f$ are always connected by Zariski Main Theorem, and they are simply-connected if and only if for each singularity $(X,x)$ of $X$ the corresponding exceptional divisor is a tree of rational curves. This happens exactly when $X$ is a $\Q$-homology manifold, and hence in such a situation  one has 
\begin{equation} \label{eq01} 
\pi_1(X)=\pi_1(\wX).
\end{equation}
(ii) The morphism $\wi_{\sharp}: \pi_1(X_0) \to \pi_1(\wX)$ is clearly a surjection for any $n$ (by an obvious transversality argument) and moreover
\begin{equation} \label{eq02} 
\wi^*:H^1(\wX,\C) \to H^1(X_{reg},\C)
\end{equation}
 is an isomorphism when $n=2$,  see Lemma 2 in \cite{DPSbull}.
We do not know whether the isomorphisms    \eqref{eq01} and  \eqref{eq02} hold for $n>2$ as well. 

(iii) Easy examples (e.g. $X=S_6$, the degree 6 cyclic covering of $\PP^2$ ramified along a Zariski sextic curve with 6 cusps situated on a conic) show that the other morphisms induced at $H^1$- level  $i^*$ and $p^*$ are not isomorphisms in general, even when $n=2$.

(iv) The monomorphism $p^*:H^1(X,\Q) \to H^1(\wX,\Q)$ implied by Theorem \ref{thm1} shows that for a projective normal variety $X$, the first cohomology $H^1(X,\Q)$ is a pure Hodge structure of weight 1, in particular the first Betti number $b_1(X)$ is even, see also Proposition \ref{prop:b1normal} below for a precise formula. For an arbitrary normal variety $X$, it follows that $H^1(X,\Q)$ has weights 1 and 2.

\end{rk}

\begin{rk} \label{rkMumford} 
There are however differences between certain properties of fundamental groups of normal varieties and  fundamental groups of smooth varieries. Perhaps the most basic is the following. If $X$ is a smooth complex (analytic) variety and $A $ is a closed subvariety of codimension $\codim A \geq 2$, then the inclusion induces an isomorphism $\pi_1(X \setminus A) \to \pi_1(X).$ To see this, one can use the usual Thom transversality theorems for smooth real manifolds.

When $X$ is normal, this result is no longer true. For instance, when $(X,x)$ is the germ of a normal surface singularity, the fundamental group of (some representative of) $X$ is trivial, due to the conic local structure of analytic sets, see \cite{BV}. But the complement $X \setminus \{x\}$, which is homotopy equivalent to the link of the singularity, has a trivial fundamental group if and only if $(X,x)$ is a smooth germ, by the celebrated result of Mumford \cite{Mu}.

\end{rk}

\section{Proof of Theorem   \ref{thm2}}

Let $p: \wX\to X$ be a resolution of singularities as above and note that by Chow's Lemma we can suppose that $\wX$ is a quasi-projective variety.

\begin{lem}\label{lemma:inject}
  If $L$ is a locally constant sheaf on $X$, then $p^*:H^1(X,L)\to H^1(\wX, \wL)$ is injective, with $\wL=p^*L$.
\end{lem}

\begin{proof}
  Note that we have by Leray Theorem 
$$H^1(\wX,\wL)=\HH^1(X, \R p_*\wL).$$
The hypercohomology group in the right hand side can be computed using the usual $E_2$-spectral sequence
$$E_2^{p,q}=H^p(X,R^q p_*\wL).$$
We have $E_{\infty}^{1,0}=E_2^{1,0}=H^1(X,R^0 p_*\wL)$ and $R^0 p_*\wL=L$ since the fibers of $p$ are connected. From this we deduce an inclusion
\begin{equation} \label{eq7} 
H^1(X,L)=   E_{\infty}^{1,0}\rightarrow H^1(\wX,\wL).
\end{equation}
which  corresponds exactly to $p^*$. 
\end{proof}

\begin{rk}\label{rk:leray}
A continuation of this argument, shows that we have an exact sequence
$$0\to H^1(X,L)\to H^1(\wX, \wL) \to H^0(X, R^1p_*\wL)\cong H^1(E,\wL)$$
  where $E$ is the exceptional divisor.
\end{rk}

This Lemma combined with the fact that $p_{\sharp}$ is surjective yields the following.
\begin{cor}\label{corincl}
We have inclusions  $H^1(X,\C)\subseteq H^1(\wX, \C)$, $H^1(X,\C^*)\subseteq H^1(\wX, \C^*)$ and $V^1(X)_1\subseteq V^1(\wX)_1$ induced by $p^*$.
\end{cor}

We discuss first the case of a non-translated component of positive dimension, i.e.  we take $V\subseteq
V^1(X)_1$. Since $\wX$ is quasiprojective, we can apply \cite{A} and find a map $g:\wX \to C$ with generic connected fiber onto
a smooth quasiprojective curve such that $g^*H^1(C,\C^*)\supseteq p^*V$. If
we can show that $g$ factors through $X$, the theorem will follow. The main step for this is the following.

\begin{prop}\label{propaux}
For any point $x\in X$,  the fibre $F_x=p^{-1}(y)$ maps to a point in $C$ under $g$. 
\end{prop}

\begin{proof}
 Since the fiber $F_x$ is connected, it is enough to show that the restriction of $g$ to any irreducible component $F$ of $F_x$ is constant. Suppose that there is a component $F$ of $F_x$ such that the restriction $g:F \to C$ is dominant. By taking a generic linear section of the quasiprojective variety $F$ and taking the smooth part of an irrducible component of this intersection, we find a smooth curve $C' \subset F$ such that the restriction $g:C' \to C$ is still dominant. To continue we need the following.

\begin{lem}\label{lemaux}
  Let $\rho \in V$ be a nontorsion element. Then $p^*\rho|_F\not= 1$.
\end{lem}

\begin{proof}[Proof of lemma]
 We know that $p^*\rho = g^*\rho'$ for some character $\rho' \in H^1(C,\C^*)$.
 The map $C'\to C$ being dominant, it follows that
  $H^1(C,\C)\to H^1(C',\C)$ is injective (see for instance Lemma 6.10 in \cite{DPS}), and thus
 $H^1(C,\C)\to H^1(F,\C)$ is also injective. It follows that the  map
 $H^1(C,\C^*)_1\to H^1(F,\C^*)_1$ 
of connected components of the identity (where we view these as 
algebraic groups and the cohomology groups $H^1(C,\C)$ and $ H^1(F,\C)$ as their Lie algebras) has finite kernel. 
Thus the pull-back of $\rho'$ to $F$ is nontrivial. 
\end{proof} 

On the other hand, we see that $p^*\rho|_F$ is the pulback of
$\rho|_y$ which is necessarily trivial. This leads to a contradiction. Therefore Proposition \ref{propaux} is proved.
\end{proof}

We thus get a mapping $h:X \to C$ such that $g=h \circ p$. Since $p$ is proper, if $K\subset C$ is
closed then $h^{-1}(K) = p(g^{-1}(K)) $ is
closed. So that $h$  is continuous and therefore locally
bounded. Moreover, $h$ is regular on the open set $X_0$ as $p$ is an isomorphism above $X_0$.
Normality of $X$ shows then that $h$ is regular everywhere. If $L_C$ is a generic rank one local system on $C$, it follows from \cite{A} that 
\begin{equation} \label{gen1} 
\dim H^1(\wX,g^*L_C)=-E(C),
\end{equation}
in particular for such components the curve $C$ should have strictly negative Euler characteristic $E(C)=-k$. This result was reproved in \cite{D}, using only the existence of the map $g$ and the topology of constructible sheaves, hence Corollary 4.7 in \cite{D} can be applied to the map $h$ and gives
\begin{equation} \label{gen2} 
\dim H^1(X,h^*L_C)=-E(C)=k.
\end{equation}
This  proves Theorem \ref{thm2} in the case of a non-translated component.

Assume now that $V$ is translated and let $V'$ be the unique the irreducible component of $V^1(\wX)$ containing $p^*V$. Then by \cite{A} we infer the existence of a map $g:\wX \to C$ as above and of a torsion character $\chi ' \in H^1(\wX,\C^*)$ such that $V'=\chi ' g^*H^1(C,\C^*)$. It follows from Theorem 5.3 in \cite{D} that this character comes from a unique homomorphism
$T(g) \to \C^*$, where 
$$T(g)=\ker \{g_*:  H_1(\wX,\Z) \to H_1(C,\Z)\}/ \im \{i_*:H_1(F,\Z) \to H_1(\wX,\Z)\}.$$
Here $i:F \to \wX$ is the inclusion of a generic fiber $F$ of $g$ and $T(g)$  is a finite Abelian group determined by the multiple fibers of $g$.
Proceeding as above, we get a factorization $g=f\circ p$. It is clear that for any $c \in C$, $f^{-1}(c)$ is a multiple fiber of multiplicity $m_c>0$ if and only if  $g^{-1}(c)$ is a multiple fiber of multiplicity $m_c>0$. It follows that the two finite groups $T(f)$ (defined exactly as above using $f$ instead of $g$) and $T(g)$ are isomorphic under $p^*$. Hence there is a torsion character $\chi \in H^1(X,\C^*)$ such that $\chi'=p^*\chi$.
The injectivity of $p^*$ then implies that $V \subset \chi f^*H^1(C,\C^*)$, and hence we have equality, since $V$ is an irreducible component of $V^1(X)$ and $ \chi f^*H^1(C,\C^*) \subset V^1(X)$.

Let now $L$ and $L'$ be the rank one local systems on $X$ and $\wX$ associated to the torsion character $\chi$, and respectively $\chi'$. Then, for a generic rank one local system $L_C$ on $C$,
 Corollary 4.7 in \cite{D}  yields the following equality
\begin{equation} \label{gen3} 
\dim H^1(X,L \otimes h^*L_C)=-E(C)+ |\Sigma (R^0h_*(L))|,
\end{equation}
and a similar formula of $\dim H^1(\wX,L' \otimes g^*L_C)$, where $\Sigma(\F)$ denotes the {\it singular support} of the constructible sheaf $\F$. Hence in order to prove the the equality
$$\dim H^1(X,L \otimes h^*L_C)=\dim H^1(\wX,L' \otimes g^*L_C),$$
it is enough to show that
\begin{equation} \label{gen4} 
\Sigma (R^0h_*(L))=\Sigma (R^0g_*(L')).
\end{equation}
Let $c \in C$ be any point, and denote by $D_c$ a small disc in $C$ centered at $c$, and set
$T_c=h^{-1}(D_c)$ and $T'_c=g^{-1}(D_c)$. Then $c \in \Sigma (R^0h_*(L))$ if and only if the restriction $L|T_c$ is non trivial, and the same for  $c \in \Sigma (R^0g_*(L'))$, see for instance Lemma 4.2 in \cite{D}. Since $T'_c=p^{-1}(T_c)$ and $L'=p^*L$ it follows that $L|T_c$ is  trivial implies that $L'|T'_c$ is trivial. 

Conversely, suppose that $L'|T'_c$ is trivial and that $E_c$ is the exceptional divisor of $p: T_c' \to T_c$. Then $L'|(T_c'\setminus E_c)$ is trivial, and hence $L|T_{c,reg}$ is trivial.
The surjectivity $\pi_1(T_{c,reg}) \to \pi_1(T_{c})$ implied by Theorem \ref{thm1} yields the triviality of $L|T_c$.

It remains to prove the following.

\begin{lem}
  The isolated components of $V^1(X)$ consist of torsion points. 
\end{lem}

\begin{proof}
We employ the method of Simpson \cite{Si1}, along 
with its extension
by Budur and Wang  \cite{BW}. We
start with the case where $X$ is projective.  
By a specialization
argument (\cite[p. 373]{Si1}) we can assume that $p:\wX\to X$ is defined over
$\bar \Q$.  Following Simpson, we introduce two spaces, the ``Betti'' space
$M_B(\wX)=Spec\,\Z[H_1(\wX,\Z)]$, and the ``de Rham'' space $M_{DR}(\wX)$ given as the
moduli space of line bundles $V$ on $\wX$ equipped with flat
connections $\nabla$.
Both of these spaces are defined over $\bar \Q$. Over $\C$, the map
\begin{equation}
  \label{eq:MDR}
\mu:M_{DR}(\wX)(\C)\to M_{B}(\wX)(\C)=H^1(\wX,\C^*)  
\end{equation}
which sends a pair $(V,\nabla)$ to its monodromy, is an analytic
isomorphism.
We can identify $V^1(X)$ with a subset of the right of \eqref{eq:MDR}, which we temporarily 
rename $V_B^1(X)$. Let $E\subset \wX$ denote the
exceptional divisor, which we assume to have normal crossings. Then for any local system $L$ on $X$, we have an isomorphism
$$H^1(X,L)\cong \ker[ H^1(\wX, L)\to H^1(E,L)]$$
obtained from Remark \ref{rk:leray}.
This  suggests the definition for the corresponding set $V_{DR}^1(X)\subset M_{DR}(\wX)$.
To $(V,\nabla)\in M_{DR}(\wX)$, we can associate its de Rham
cohomology
$$
H^i(\wX, (V,\nabla)) =\mathbb{H}^i(\wX,\Omega_{\wX}^\bullet\otimes
V,\nabla) = H^i(\R\Gamma(\Omega_{\wX}^\bullet\otimes
V,\nabla))
$$
$$
H^i(E, (V,\nabla)) =H^i(Tot^\bullet\R\Gamma(\Omega_{E_\bullet}^\bullet\otimes
V,\nabla))
$$
where $E_\bullet$ is the standard simplicial resolution of $E$ (as
given  for example in \cite[\S 4]{GS}).
Let
$V_{DR}^1(X)\subset M_{DR}(\wX)$  consist of those pairs $(V,\nabla)$
such that $(V,\nabla)|_E =(\OO_E,d)$ and 
\begin{equation}
  \label{eq:kerH1}
\ker [H^1(\wX, (V,\nabla))\to H^1(E, (V,\nabla))]\not=0  
\end{equation}
Both sets $V_{B}^1(X)$ and $V_{DR}^1(X)$ are defined over $\bar \Q$ and
correspond under \eqref{eq:MDR}. Therefore we may apply  \cite[Cor
3.5]{Si1} to conclude that an isolated point $S_B\subset V_B^1(X)$ is torsion.

When $X$ is only quasiprojective, we choose
a  normal compactification $Y\supset X$ and a
smooth projective compactification $\wY\supset \wX$ such that $D=
D_1+\ldots +D_n= \wY-\wX$ plus $E$ is a divisor
with normal crossings.  We can proceed as above, but now
\eqref{eq:MDR} is replaced by  the diagram 
$$
\xymatrix{
 M_{DR}(\wY)\ar[r]\ar[d]^{\mu} & M_{DR}(\wY/D)\ar[r]^{res}\ar[d]^{\mu} & \C^n\ar[d]^{e^{-2\pi i}} \\ 
 M_{B}(\wY)\ar[r] & M_B(\wX)\ar[r]^{ev} & (\C^*)^n
}
$$
of \cite{BW}, where $M_{DR}(\wY/D)$ is the moduli space of line bundles
on $\wY$ with connections with logarithmic singularities along
$D$. The map $res$ assigns the vector of residues to a connection, and
$ev$ evaluates a local system on loops around the $D_i$.
The sets  $V^1_B(X)\subset M_B(\wX)$ and $V^1_{DR}(X)\subset M_{DR}(\wY/D)$
can be defined  as above, where $(V,\nabla)|_E =(\OO_E,d)$ as before,
but  the second condition \eqref{eq:kerH1} is replaced by
$$\ker[\mathbb{H}^1(\wY,\Omega_{\wY}^\bullet (\log D)\otimes
V,\nabla)\to {H}^1(Tot^\bullet(\R\Gamma(\Omega_{E_\bullet}^\bullet (\log D)\otimes
V,\nabla)))]\not=0$$
Both $V^1_B(X)$ and $V^1_{DR}(X)$ are defined over $\bar
\Q$, in a way compatible with the evident $\bar\Q$ structures on
$\C^n$ and $(\C^*)^n$. Deligne's comparison theorem \cite[II, Thm
6.10]{De1} shows that $\mu(V^1_{DR}(X)-res^{-1}\N^n) = V_B^1(X)$. 
 Let $\tau\in V^1_B(X)$ be an
isolated point. Then it follows that it is defined over $\bar \Q$ and  is the image
of a $\bar \Q$ point of $V_{DR}^1(X)$.
Therefore the Gelfond-Schneider theorem
implies  that $\tau$ maps to an $N$-torsion
point under $ev$ for some $N\ge 1$. Therefore $N\tau\in
V^1_B(Y)=V_B^1(X)\cap M_B(\wY)$.
This is torsion by the previous case of the lemma.

\end{proof}

\begin{rk} \label{rktorsion} 
If $L$ is an isolated point in some characteristic variety $V^1_k(X)$, the following questions seem to be open.

\noindent (1) Is $p^*(L)$ an isolated point in some characteristic variety of $\wX$?

\noindent (2) What is the relation between $\dim H^1(X,L)$ and $\dim H^1(\wX,p^*L)$?

\end{rk} 

\section{Proof of Theorem  \ref{thm3} and of Corollary \ref{maincor2}}
The proof of Theorem  \ref{thm3} follows closely the proof of Theorem 5.8 in \cite{H3}, which itself gives a new proof of Corollary (10.3) in Morgan's paper \cite{Mo}. 

Though both of these results discuss only smooth varieties, the construction by Hain of a mixed Hodge structure (MHS for short) on the (rational) Malcev Lie algebra $\g=\g(X,x)$ associated to $\pi_1(X,x)$ for any complex algebraic variety in \cite{H1}, \cite{H2} allows us to proceed as follows.

By Remark \ref{rk1}, (iv) we know that $H^1(X)$ is a pure Hodge structure of weight 1 for $X$ normal projective (resp. has weights 1 and 2 for $X$ normal quasi projective), and hence by duality, the homology group $H_1(X)$ is a pure structure of weight $-1$ for $X$ projective (resp. has weights $-1$ and $-2$ for $X$ quasi projective).

There is a canonical homomorphism $H^\ast(\g) \to H^\ast(X)$ which is an
isomorphism in degree 1 and injective in degree 2. One can prove that it is a
morphism of MHS in all degrees. However, all we need to know here is that it is a
morphism in degrees $\le 2$. We prove the last statement because it is much simpler. That it is an isomorphism of
MHS in degree 1 is a direct consequence of \cite[Thm.~6.3.1(c)]{H1}.
That it is a morphism in degree 2 follows by the argument given in
\cite[Prop.~7.1]{H3}. In the current case, the argument, though
formally the same, is simpler as the coefficient module is the trivial local
system $\Q_X$.

Since $\g$ is topologically generated as a Lie algebra by the image of
any section of the surjection $\g \to H_1(X,\Q)$, it follows that $\g$ is a Hodge Lie algebra all of whose weights are negative. Then Proposition 5.2 in \cite{H3} implies that there is a canonical Lie algebra isomorphism
\begin{equation} \label{gr} 
\g_{\C}= \prod_{j\geq1}\gr^W_{-j}\g_{\C},
\end{equation}
and hence the first claim in Theorem 5.8 in \cite{H3} holds for $X$ normal.


Using now Corollary 5.7 in \cite{H3}, we get the following version of the second claim in Theorem 5.8 in \cite{H3}.

\begin{lem}\label{keylem}
 There is a morphism of graded vector spaces
$$\delta: \gr^W_*H_2(X,\C) \to \bL(\gr^W_*H_1(X,\C))$$          
such that
$$\gr^W_*\g_{\C}=\bL(\gr^W_*H_1(X,\C))/(\delta(\gr^W_*H_2(X,\C))).$$
Here $\bL(E)$ denotes the free Lie algebra spanned by a finite dimensional $\C$-vector space $E$, see Chapter 3 in \cite{ABC} for details.
\end{lem}

Let us study the image of $\delta$ in Lemma \ref{keylem} for $X$ normal. We know that $H_2(X,\C)$ has weights $0,-1,-2,-3$ and $-4$. 
The degree $2$ part of $\bL(\gr^W_*H_1(X,\C))$  has weights strictly less than $-1$ because it is a quotient of $\otimes^2 H_1(X,\C)$.
Therefore the graded pieces $\gr^W_0H_2(X,\C)$ and  $\gr^W_{-1}H_2(X,\C)$  are mapped to zero under $\delta$.
Hence the image of $\delta$ coincides with the image of $\bigoplus_{j=2,3,4}\gr^W_{-j}H_2(X,\C)$ which consists of elements of weight $-2$, $-3$ and $-4$ in $\bL(\gr^W_*H_1(X,\C))$. 

When $X$ is projective, it is enough to show that the image of $\delta$ consists of quadratic elements, i.e. elements in $\otimes^2 H_1(X,\C)$. Since $X$ is projective, we know that $H_2(X,\C)$ has weights $0,-1$ and $-2$. The graded pieces $\gr^W_0H_2(X,\C)$ and $\gr^W_{-1}H_2(X,\C)$
are mapped to zero under $\delta$ for the same reasons as above.
Hence the image of $\delta$ coincides with the image of $\gr^W_{-2}H_2(X,\C)$ which consists of elements of weight $-2$ in $\bL(\gr^W_*H_1(X,\C))=\bL(H_1(X,\C))$. Such elements are clearly contained in $\otimes^2 H_1(X,\C)$,
This completes our proof of Theorem  \ref{thm3}.

\bigskip

To prove Corollary \ref{maincor2},
recall  that the  (finitely presentable) fundamental group $\pi_1(X,x)$ is 1-formal if and only if its Malcev Lie algebra $\g_{\C}$ is isomorphic to the degree completion of a quadratically presented Lie algebra, see for instance Proposition 3.20 in \cite{ABC}.

The above proof yields the following result. For stronger or related  statements refer to Theorem 5 in \cite{H0} (for a proof see Theorem 12.8 in \cite{H-1}) showing the formality of $\Q$-manifolds, and Theorem 11 in \cite{Bis}, showing the formality of $V$-manifolds.

\begin{cor} \label{maincor3} 
The fundamental group $G=\pi_1(X)$ of a  projective variety $X$ which is a $\Q$-manifold is 1-formal.
\end{cor} 
Indeed, if $X$ is a proper variety which is a $\Q$-manifold, then the Deligne MHS on $H^k(X,\Q)$ is in fact pure of weight $k$ for any $k$, see Theorem 8.2.4 in \cite{HodgeIII}. It follows that the above proof for Corollary \ref{maincor2} works in this new setting as well.

\begin{rk} \label{rk2} 
There are examples of smooth varieties $X$ such that $H^1(X,\Q)$ is a pure Hodge structure of weight 1 and $\pi_1(X,x)$ is not 1-formal, see for instance section 10 in \cite{DPS}.

\end{rk}  

 \section{Proof of Proposition \ref{prop:spPi} and of Theorems  \ref{thm:smoothing1} and  \ref{thm:smoothing2}}

First we prove Proposition \ref{prop:spPi}.
The map $\pi_1(X^*)\to \pi_1(X)$ is surjective by Theorem \ref{thm1}. The  image of $\pi_1( \partial \Delta)\to \pi_1(X)$ induced by $\sigma$ is trivial. Since we have a split extension $1\to \pi_1(X_t)\to \pi_1(X^*)\to \pi_1(\partial \Delta)\to 1$, the proposition follows.

The following result is needed in the  proof  of  Theorem  \ref{thm:smoothing1}, but we think it has an independent interest as well.
\begin{prop}\label{prop:b1normal}
  If $V$ is a normal projective variety, the first Betti number $b_1(V)$ of $V$ satisfies the equality
$$b_1(V)=2\dim H^1(V,\OO_V).$$
\end{prop}

\begin{proof}
  Let $\pi:\tilde V\to V$ be a resolution of singularities. Then if
  $L$ is a line bundle on $V$, we see by the projection formula that
  $L= \pi_*\pi^*L = \pi_*\OO_{\tilde V}\otimes L = L$ since
  $\pi_*\OO_{\tilde V}=\OO_V$ by normality. Therefore
  $\pi^*:Pic^0(V)\to Pic^0(\tilde V)$ is injective. Consequently,
  $Pic^0(V)$ is an abelian variety. The exponential sequence
  $1\to \Z\to \OO_V\to \OO_V^*\to 1$ \cite[p 142]{GR}, implies
$$Pic^0(V) = \frac{H^1(V,\OO_V)}{H^1(V, \Z)}$$
It follows that $H^1(V, \Z)$ is a lattice in $H^1(V,\OO_V)$. This proves the proposition.
\end{proof}

Now we turn to the proof of Theorem \ref{thm:smoothing1}.
  Proposition \ref{prop:spPi} shows that the specialization map
  $H_1(X_t,\Z)\to H_1(X_0,\Z)$ is surjective. Let $K$ denote the
  kernel. Dualizing gives an exact sequence
  \begin{equation}
    \label{eq:H1X0}
0\to H^1(X_0,\Z)\to H^1(X_t,\Z)\to Hom(K,\Z)    
  \end{equation}
Thus the Betti numbers satisfy $b_1(X_0)\le  b_1(X_t)$.
Since all the above groups in \eqref{eq:H1X0} are  torsion free, to
show (1) it is  enough to prove that
the $b_1(X_0)=b_1(X_t)$.
To see this, note that $\dim H^1(X_t,\OO_{X_t})$ is upper semicontinuous
as function of $t$ \cite{Ha}. When combined with Proposition
\ref{prop:b1normal}, we get an inequality in the opposite direction
$b_1(X_0)\ge b_1(X_t)$. Therefore assertion (1) follows.
The image of the specialization map  clearly lies in the invariant part $H^1(X_t,\Z)^{T}$, where $T$ denotes local monodromy. 
This implies (2). (3)
is consequence of (1) and proposition \ref{prop:spPi}.

We now turn to the last assertion (4). Given a rank one local system
$L$ on $X_0$, we can view it as a local system on $X$, and therefore
on $X_t$ by restriction. Proposition \ref{prop:spPi} implies that the map
$$H^1(X_0,L)\to H^1(X_t,L)$$
is injective because it can be identified with an edge map for Hochschild-Serre spectral sequence, see for instance \cite{Brown}, p. 171 (where the homology case is treated and hence an epimorphism is obtained). One has to remark that  the character $\rho_L$ associated to $L$
can be regarded as being defined on $H_1(X_0,\Z)$, has a lift to $H_1(X_t,\Z)$ and hence has a trivial restriction to the kernel $K$.

This immediately gives an inclusion $V^1(X_0)\subseteq
V^1(X_t)$.  We need the reverse inclusion. By
Beauville's structure  theorem \cite{B},  the set $V^1(X_t)$ is
Zariski closed, and torsion points are Zariski dense in each
component. Thus we may suppose that $L\in V^1(X_t)\cap H^1(X_0,\C^*)$ is $n$-torsion
and then show that $L\in V^1(X_0)$.  
 By Kummer theory, we may choose an \'etale cover $\tilde X=
 Specan(\bigoplus_{i=0}^{n-1} L^{\otimes i})\stackrel{p}{\to} X$ such
 that $p^*L$ is trivial \cite[\S 17]{bpv}. We have
 $$H^1(\tilde X, \C) = \bigoplus H^1(X, L^{\otimes i})$$
 We can pick out the individual summands on the right by decomposing
 the left side into irreducibles under the action of the  Galois group
 $\Z/n\Z$ of the covering. By applying (1) to $\tilde X\to \Delta$, we find that 
$H^1(\tilde X_0,\C)\cong H^1(\tilde X_t, \C)$. Since restriction is clearly $\Z/n\Z$ equivariant, we obtain
$H^1(X_0, L)\cong H^1(X_t, L)$. Therefore $L\in V^1(X_0)$.

\medskip
Now we pass to the proof of Theorem \ref{thm:smoothing2}.
By Theorem~\ref{thm2}, we can find torsion
characters $\chi_i$ and maps to curves $\pi_{i,0}':X_0\to C_{i,0}'$
such that 
$$V^1(X_0)=\bigcup_{i=1}^n
\chi_i\,{\pi_{i,0}'}^*H^1(C_{i,0}',\C^*)\cup
\bigcup_{i=n+1}^N\{\chi_i\}$$
Let  $V_i= {\pi_{i,0}'}^*H^1(C_{i,0}',\C^*)$  and  let $V_{i,t}\subset
H^1(X_t,\C^*)$ denote its image
under specialization. 
By theorem~\ref{thm:smoothing1}, we have 
$$V^1(X_t)\cap H^1(X_0,\C^*) =\bigcup_{i=1}^n
\chi_iV_{i,t}\cup \bigcup_{i=n+1}^N\{\chi_i\}$$
for $t\not=0$.
By applying  theorem~\ref{thm2} to $X_t$, we see that 
$V_{i,t}$ is necessarily the preimage of a map to a curve
$\pi_{i,t}:X_t\to C_{i,t}$. The remaining issue is to show that these
curves $C_{i,t}$ and the
maps $\pi_{i,t}$  fit together into a family, and this extends to $t=0$. It is convenient to first
linearize these objects by setting
\begin{equation}
  \label{eq:U}
  U_{i,t}=\pi_{i,t}^*H^1(C_{i,t}',\Z)
\end{equation}
This is a sub Hodge structure of $H^1(X_t,\Z)$. 
This determines an abelian variety
$$A_{i,t} = \frac{U_{i,t,}^*\otimes \C}{U_{i,t}^*+F^{-1}U_{i,t,}^*\otimes \C}$$
which is a quotient of  the Albanese $Alb(X_t)$. We have a map
$a_t:X_t\to A_{i,t}$ given by composing the Albanese map $X_t\to
Alb(X_t)$, determined by the base point $\sigma(t)$, with the quotient
map. We recover the curve  $C_{i,t}$ from $U_{i,t}$ by taking the
$im(a_t)$.

We can characterize the complexification $U_{i,t,\C}$  as the connected
component of $\exp^{-1}V_{i,t}$  containing $0$, where
$\exp:H^1(X,\C)\to H^1(X,\C^*)$ is  the exponential map. This also
determines the lattice via $U_{i,t}=U_{i,t,\C}\cap H^1(X_t,\Z)$.  This
new characterization shows that $U_{i,t}$ is the image of $U_{i,0}$
under specialization.

Viewing local systems as
bundles, we can see that
$$H^1(X_0,\Z)\times \Delta^*\cong \bigcup_{t\in \Delta^*} H^1(X_t,\Z) =R^1f_*\Z|_{\Delta^*}
$$
is constant, where the isomorphism is given by specialization.
The previous discussion shows that 
$$\mathcal{U}_i= \bigcup_{t\not=0} U_{i,t}\subset \bigcup H^1(X_t,\Z)$$ 
forms a sublocal system. 
 Since  the fibres are sub Hodge
structures of $H^1(X_t,\Z)$, it follows
that $\mathcal{U}_i$ is a sub-variation of Hodge structure of
$=R^1f_*\Z|_{\Delta^*}$, which we denote by $\mathcal{U}_i$.  We can now construct a
bundle 
$$A_i= \frac{\mathcal{U}_{i,\C}^*}{\mathcal{U}_i^*+F^{-1}\mathcal{U}_{i,\C}^*}$$
of Abelian varieties over $\Delta$.
When restricted to $\Delta^*$, this is a 
quotient of the relative Albanese $Alb(X^*/\Delta^*)$. Let $a:X^*\to
Alb(X^*/\Delta^*)\to A_i$ be the composite, where the
first map is a relative version of the Albanese map. We can see that
the 
fibre over $t$ is precisely $a_t:X_t\to A_{i,t}$ constructed above. 
Let $C_i^* = image(a)\subset A_i|_{\Delta^*}$. 
Then the fibres of $C_i^*$
are exactly the $C_{i,t}, t\in \Delta^*$. 

To complete  the proof, we have to show that the maps $\pi_i^*:X_i^*\to C_i^*$
constructed above extend. Since the monodromy of the family 
$C_i^*$ is trivial, and in particular unipotent, we can conclude first
of all that this has an extension $C_i$ to a stable curve over $\Delta$,
and secondly that $C_i$ is in fact smooth over $0$. Now let
$\Gamma\subset X\times C_i$ denote the closure of the graph of
$\pi_i^*$. Let $p:\Gamma\to X$ and $q:\Gamma\to C_i$ denote the projections. Given $x_0\in X_0$,
we claim that $p^{-1}(x_0)$ maps to a point  in $C_i$ under $q$.  Suppose not,
then arguing exactly as in the
proof of proposition \ref{propaux}, we would see that the induced map
$r:H^1(C_{i,0},\C)\to H^1(p^{-1}(x_0),\C)$ would have to be nonzero.
We can choose an analytic arc passing
through $x_0$ and not contained in $X_0$.
 After normalizing  and shrinking the arc, we obtain a holomorphic function $g:\Delta\to X$ such that $g(0)=x_0$
and  $g\circ f:\Delta\to \Delta$ is a finite cover ramified at
$0$. Let $t$ be nonzero and sufficiently small, then we have a diagram
$$
\xymatrix{
 H^1(C_0,\C)\ar[r]^{r}\ar[d]^{\cong} & H^1(p^{-1}(x_0),\C)\ar[d] \\ 
 H^1(C_t,\C)\ar[r] & H^1(p^{-1}((g\circ f)^{-1}(t)),\C)=0
}
$$
This shows that the map  $r=0$.  So the claim is proved, and it implies
that $\Gamma$ is the graph of a map $\pi_i:X\to C_i$ extending
$\pi_i^*$.  By chasing the diagram
$$\xymatrix{
 H^1(C_{i,0},\C^*)\ar[r]\ar[d]^{\cong} & H^1(X_0,\C^*)\ar[d] \\ 
 H^1(C_{i,t},\C^*)\ar[r] & H^1(X_t,\C^*)
}
$$
We can see that $\pi_{i,0}^*H^1(C_{i,0},\C^*)=V_i$, and this completes
the proof.


\end{document}